\documentclass[10pt,reqno]{article}%{article}
\usepackage{fullpage} % decreases margins
\usepackage{amssymb} % AMS-LaTeX + Theorems + Symbols
\usepackage{graphicx} % to include graphics files
\usepackage{amsmath} % AMS-LaTeX + Theorems + Symbols
\usepackage{amsthm,amssymb,latexsym} % AMS-LaTeX + Theorems + Symbols
\usepackage{amstext, amsfonts,amsbsy} % not sure why I need it
\usepackage{enumerate} % easy customization of enumeration
\usepackage{upgreek} % for \upalpha, etc.
\usepackage{color} % adds color
\usepackage{accents} % for \undertilde
\usepackage{mathtools} % for dcases, etc.
\usepackage{float} % forcing figure floats
\usepackage{tikz}
\usetikzlibrary{matrix,graphs,arrows,positioning,calc,decorations.markings,shapes.symbols}
\usepackage[pdftex,bookmarks,colorlinks,breaklinks]{hyperref}  % PDF hyperlinks, with coloured links
\definecolor{dullmagenta}{rgb}{0.4,0,0.4}   % #660066
\definecolor{darkblue}{rgb}{0,0,0.4}
\hypersetup{linkcolor=red,citecolor=blue,filecolor=dullmagenta,urlcolor=darkblue} % coloured links
\usepackage{eucal}
\usepackage{booktabs}
\newtheorem{theorem}{Theorem}%[section]

\newtheorem{lemma}[theorem]{Lemma}

\newcommand{\dPain}[1]{\text{P}\left(\mathrm{#1}\right)}

\theoremstyle{definition}

\theoremstyle{remark}

\numberwithin{equation}{section}
\begin{document}
{\noindent\Large\bf Recurrence coefficients for the time-evolved Jacobi weight and discrete Painlev\'e equations on the $D_{5}$ Sakai surface}
\medskip
\begin{flushleft}
\textbf{Mengkun Zhu}\\
School of Mathematics and Statistics, Qilu University of Technology (Shandong Academy of Sciences)
Jinan 250353, China\\
E-mail: \href{mailto:zmk@qlu.edu.cn}{\texttt{zmk@qlu.edu.cn}}\\[5pt]

\textbf{Siqi Chen}\\
School of Mathematics and Statistics, Qilu University of Technology (Shandong Academy of Sciences)
Jinan 250353, China\\
E-mail: \href{mailto:chen_siqi0301@163.com}{\texttt{chen\_siqi0301@163.com}}\\[5pt]

\textbf{Xuhao Zhang}\\
School of Mathematics and Statistics, Qilu University of Technology (Shandong Academy of Sciences)
Jinan 250353, China\\
E-mail: \href{zhanghaotimmy@163.com}{\texttt{zhanghaotimmy@163.com}}\\[5pt]

\end{flushleft}
\begin{abstract}
In this paper, we focus on the relationship between the d-$\dPain{A_{3}^{(1)}/D_{5}^{(1)}}$ equations and a time-evolved Jacobi weight,
\begin{equation*}
w(x)=x^{\alpha}(1-x)^{\beta}\mathrm{e}^{-sx}, \quad x\in[0,1], \quad \alpha,\beta > -1, \quad s>0.
\end{equation*}
From the perspective of Sakai's geometric theory of Painlev\'e equations, we derive that a recurrence relation closely related to the recurrence coefficients of monic polynomials orthogonal with $w(x)$ is equivalent to the standard d-$\dPain{A_{3}^{(1)}/D_{5}^{(1)}}$ equation.
\end{abstract}
\section{Introduction} % (fold)
\label{Introduction}
The study of recurrence coefficients of orthogonal polynomials for varying weights is often related to Painlev\'e equations. For example, Clarkson and Jordaan \cite{CJ:2014:TRBSLPATFPE} examined the semi-classical Laguerre weight $w(x)=x^{\lambda}\mathrm{e}^{-x^2+sx}$, with $x\in\mathbb{R}^{+},  \lambda > -1, s\in\mathbb{R}$, and showed that its recurrence coefficients satisfy differential equations associated with the $\mathrm{P_{IV}}$ equation. For additional relevant research, refer to \cite{Mag:1995:PTDEFTRCOSOP}, \cite{BV:2010:DPEFRCOSLP} and \cite{CI:2010:PAASLSIHRMEI}, see also a recent monograph \cite{VA:2018:OPAPE} and its references. Sakai \cite{Sak:2001:RSAWARSGPE} proposed a geometric-theoretic classification of both continuous and discrete Painlev\'e equations, which completely classified all possible configuration spaces for discrete Painlev\'e dynamics as families of specific rational algebraic surfaces\textemdash generalized Halphen surfaces.

In this paper, building on Sakai's geometric framework, we focus on the time-evolved Jacobi weight,
\begin{equation}\label{eq:wf}
w(x)=w(x;s)=x^{\alpha}(1-x)^{\beta}\mathrm{e}^{-sx}, \quad x\in[0,1], \quad \alpha,\beta > -1, \quad s>0.
\end{equation}
Let $P_n(x)$ be the monic orthogonal polynomials of degree $n$ with respect to the weight \eqref{eq:wf}, i.e.
\begin{equation*}
\int_{0}^{1}P_{m}(x;s)P_{n}(x;s)w(x;s)\mathrm{d}x=h_{n}(s)\delta_{m,n},\qquad m,n=0,1,2,...,
\end{equation*}
where
\begin{equation*}
P_{n}(x;s)=x^{n}+\mathrm{p}(n,s)x^{n-1}+\cdots.
\end{equation*}
The monic orthogonal polynomials satisfy the three term recurrence relation
\begin{equation*}
xP_{n}(x;s)=P_{n+1}(x;s)+\alpha_{n}P_{n}(x;s)+\beta_{n}P_{n-1}(x;s),
\end{equation*}
with initial conditions $P_{0}(x;s)=1$, $\beta_{0}P_{-1}(x;s)=0$.

The lowering and raising ladder operators for $P_n(x)$ with respect to \eqref{eq:wf} are given by
\begin{align*}
P_{n}'(x)    &=-B_{n}(x)P_{n}(x)+\beta_{n}(s)A_{n}(x)P_{n-1}(x),\\
P_{n-1}'(x)  &=[B_{n}(x)+v'(x)]P_{n-1}(x)-A_{n-1}(x)P_{n}(x),
\end{align*}
where
\begin{equation}\label{eq:AB}
A_{n}(x)    =\frac{R_{n}(s)}{x}+\frac{s-R_{n}(s)}{x-1},\qquad
B_{n}(x)    =\frac{r_{n}(s)}{x}-\frac{n+r_{n}(s)}{x-1},
\end{equation}
and
\begin{align*}
R_{n}(s)    &:=\frac{\alpha}{h_{n}}\int_{0}^{1}P_{n}^2(y)y^{\alpha-1}(1-y)^{\beta}\mathrm{e}^{-sy}\mathrm{d}y,\\
r_{n}(s)    &:=\frac{\alpha}{h_{n-1}}\int_{0}^{1}P_{n}(y)P_{n-1}(y)y^{\alpha-1}(1-y)^{\beta}\mathrm{e}^{-sy}\mathrm{d}y.
\end{align*}
The functions $A_{n}(x)$ and $B_{n}(x)$ are not independent but satisfy the following compatibility conditions,
\begin{align}
B_{n+1}(x)+B_n(x)&=\left(x-\alpha_n(s)\right)A_n(x)-v'(x),\tag{$S_1$}\label{$S_1$}\\
1+\left(x-\alpha_n(s)\right)\left(B_{n+1}(x)-B_{n}(x)\right)&=\beta_{n+1}(s)A_{n+1}(x)-\beta_{n}(s)A_{n-1}(x),\tag{$S_2$}\\
B_n^2(x)+v'(x)B_n(x)+\sum_{j=0}^{n-1}A_{j}(x)&=\beta_n(s)A_n(x)A_{n-1}(x),\tag{$S_2'$}\label{$S_2'$}
\end{align}
where $v(x)=-\ln w(x)$. In theorem 2.6 of \cite{ZBCZ:2018:COMDOTJUEPAE}, Zhan et al. substitute \eqref{eq:AB} into \eqref{$S_1$} and \eqref{$S_2'$}, after a series of simplifications, they obtain the following equations,
\begin{align}
&s(r_{n+1}+r_{n})=R_{n}^2-(2n+1+\alpha+\beta+s)R_{n}+s\alpha,\label{eq:R}\\
&n(n+\beta)+(2n+\alpha+\beta)r_{n}=(r_{n}^2-\alpha r_{n})\left(\frac{s^2}{R_{n}R_{n-1}}-\frac{s}{R_{n}}-\frac{s}{R_{n-1}}\right).\label{eq:r}
\end{align}

According to \eqref{eq:R} and \eqref{eq:r}, we introduce a transformation
\begin{equation*}
x_{n}(s):=\frac{1}{s}-\frac{1}{R_{n-1}(s)}, \qquad y_{n}(s):=-r_n(s),
\end{equation*}
which leads to the recurrence relation
\begin{equation}\label{eq:xyn-evol}
	\left\{ \begin{aligned}
		 x_{n}x_{n+1} &= \frac{y_{n}^2-(2n+\beta)y_{n}+n(n+\beta)}{s^2  (y_{n}^2+\alpha y_{n})},\\
		 y_{n}+y_{n-1} &=  - \frac{\alpha s^2x_{n}^2+s(2n-1-\alpha+\beta+s)x_{n}-2n-\beta+1}{(1-sx_{n})^2}.
	\end{aligned}\right.
\end{equation}

In section~\ref{Preliminaries_and_the_main_result} and \ref{Proof_of_the_main_result} of this paper, we establish a relationship between this recurrence relation and the standard  d-$\dPain{A_{3}^{(1)}/D_{5}^{(1)}}$ equation.
\section{Preliminaries and the main result} % (fold)
\label{Preliminaries_and_the_main_result}
To make this paper self-contained, we briefly describe some basic geometric data for the standard discrete Painlev\'e equations in the $D_5^{(1)}$ surface family, based on \cite{KajNouYam:2017:GAOPE} and the appendix of \cite{HDC:2020:GPITLUEADPE}.

The standard surface root basis and the standard symmetry sub-lattices depicted in Figure \ref{fig:d-roots-d51-a31}.
\begin{figure}[ht]
\begin{equation}\label{eq:d-roots-d51-a31}			
	\raisebox{-32.1pt}{\begin{tikzpicture}[
			scale=0.75, % 缩小第一个tikzpicture
			elt/.style={circle,draw=black!100,thick, inner sep=0pt,minimum size=2mm}]
		\path 	(-1,1) 	node 	(d0) [elt, label={[xshift=-10pt, yshift = -10 pt] $\delta_{0}$} ] {}
		        (-1,-1) node 	(d1) [elt, label={[xshift=-10pt, yshift = -10 pt] $\delta_{1}$} ] {}
		        ( 0,0) 	node  	(d2) [elt, label={[xshift=-10pt, yshift = -10 pt] $\delta_{2}$} ] {}
		        ( 1,0) 	node  	(d3) [elt, label={[xshift=10pt, yshift = -10 pt] $\delta_{3}$} ] {}
		        ( 2,1) 	node  	(d4) [elt, label={[xshift=10pt, yshift = -10 pt] $\delta_{4}$} ] {}
		        ( 2,-1) node 	(d5) [elt, label={[xshift=10pt, yshift = -10 pt] $\delta_{5}$} ] {};
		\draw [black,line width=1pt ] (d0) -- (d2) -- (d1)  (d2) -- (d3) (d4) -- (d3) -- (d5);
	\end{tikzpicture}} \quad
			\begin{alignedat}{1}
			\delta_{0} &= \mathcal{E}_{1} - \mathcal{E}_{2}, \\
            \delta_{1} &= \mathcal{E}_{3} - \mathcal{E}_{4}, \\
            \delta_{2} &= \mathcal{H}_{f} - \mathcal{E}_{1} - \mathcal{E}_{3}, \\
            \delta_{3} &= \mathcal{H}_{g} - \mathcal{E}_{5} - \mathcal{E}_{7},\\
            \delta_{4} &= \mathcal{E}_{5} - \mathcal{E}_{6},\\
            \delta_{5} &= \mathcal{E}_{7} - \mathcal{E}_{8}.
			\end{alignedat} \qquad
\raisebox{-32.1pt}{\begin{tikzpicture}[
			scale=0.75, % 缩小第二个tikzpicture
			elt/.style={circle,draw=black!100,thick, inner sep=0pt,minimum size=2mm}]
		\path 	(-1,1) 	node 	(a0) [elt, label={[xshift=-10pt, yshift = -10 pt] $\alpha_{0}$} ] {}
		        (-1,-1) node 	(a1) [elt, label={[xshift=-10pt, yshift = -10 pt] $\alpha_{1}$} ] {}
		        ( 1,-1) node  	(a2) [elt, label={[xshift=10pt, yshift = -10 pt] $\alpha_{2}$} ] {}
		        ( 1,1) 	node 	(a3) [elt, label={[xshift=10pt, yshift = -10 pt] $\alpha_{3}$} ] {};
		\draw [black,line width=1pt ] (a0) -- (a1) -- (a2) --  (a3) -- (a0);
	\end{tikzpicture}} \quad
			\begin{alignedat}{1}
			\alpha_{0} &= \mathcal{H}_{g} - \mathcal{E}_{1} - \mathcal{E}_{2}, \\
			\alpha_{1} &= \mathcal{H}_{f} - \mathcal{E}_{5} - \mathcal{E}_{6}, \\
            \alpha_{2} &= \mathcal{H}_{g} - \mathcal{E}_{3} - \mathcal{E}_{4},\\
            \alpha_{3} &= \mathcal{H}_{f} - \mathcal{E}_{7} - \mathcal{E}_{8}.
			\end{alignedat}
\end{equation}
	\caption{The standard surface (left) and symmetry (right) root basis for the $D_{5}^{(1)}$ surface.}
	\label{fig:d-roots-d51-a31}	
\end{figure}

Consider the point configuration and its associated Sakai surface illustrated in Figure~\ref{fig:surface-d5}. Here, the base points are formulated using the root variables $a_{0}$, $a_{1}$, $a_{2}$ and $a_{3}$, which are subject to the normalization condition $a_{0}+a_{1}+a_{2}+a_{3}=1$, these are given as follows:
\begin{equation}
\begin{alignedat}{2}
	&p_{1}\left(F=\frac{1}{f}=0,g=-t\right)&&\leftarrow p_{2}\left(u_{1}=\frac{1}{f}=0,v_{1}=f(g+t)=-a_{0}\right),\\
	&p_{3}\left(F=\frac{1}{f}=0,g=0\right)&&\leftarrow p_{4}\left(u_{3}=\frac{1}{f}=0,v_{3}=fg=-a_{2}\right),\\
    &p_{5}\left(f=0,G=\frac{1}{g}=0\right)&&\leftarrow
p_{6}\left(U_{5}=fg=a_{1},V_{5}=\frac{1}{g}=0\right),\\
    &p_{7}\left(f=1,G=\frac{1}{g}=0\right)&&\leftarrow
p_{8}\left(U_{7}=(f-1)g=a_{3},V_{7}=\frac{1}{g}=0\right).
\end{alignedat}
\end{equation}
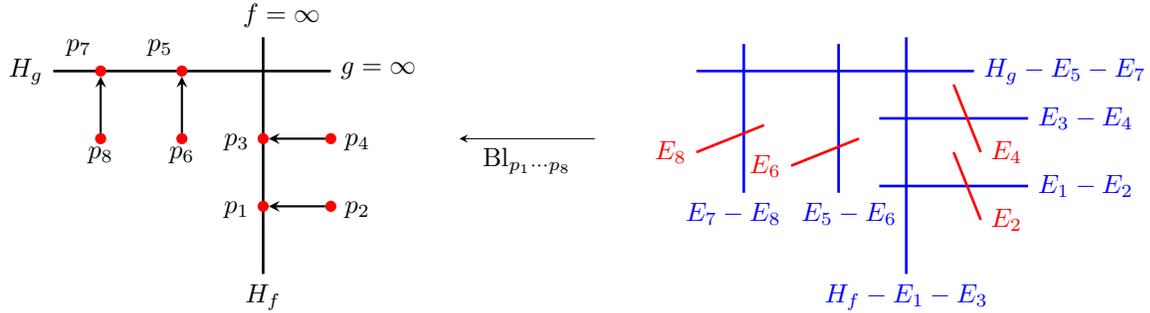
\begin{figure}[ht]
	\begin{tikzpicture}[>=stealth,basept/.style={circle, draw=red!100, fill=red!100, thick, inner sep=0pt,minimum size=1.2mm}, scale=0.9]
		\begin{scope}[xshift = -1cm]
			\draw [black, line width = 1pt] 	(3.6,2.5) 	-- (-0.5,2.5)	node [left] {$H_{g}$} node[pos=0, right] {$g=\infty$};
			\draw [black, line width = 1pt] 	(2.6,3) -- (2.6,-0.5)		node [below] {$H_{f}$} node[pos=0, above, xshift=7pt] {$f=\infty$};
		
			\node (p1) at (2.6,0.5) [basept,label={[xshift=-10pt, yshift = -10 pt] $p_{1}$}] {};
			\node (p2) at (3.6,0.5) [basept,label={[xshift=10pt, yshift = -10 pt] $p_{2}$}] {};
			\node (p3) at (2.6,1.5) [basept,label={[xshift=-10pt, yshift = -10 pt] $p_{3}$}] {};
			\node (p4) at (3.6,1.5) [basept,label={[xshift=10pt, yshift = -10 pt] $p_{4}$}] {};
			\node (p5) at (1.4,2.5) [basept,label={[above left] $p_{5}$}] {};
			\node (p6) at (1.4,1.5) [basept,label={[xshift=0pt, yshift = -15 pt] $p_{6}$}] {};
			\node (p7) at (0.2,2.5) [basept,label={[above left] $p_{7}$}] {};
			\node (p8) at (0.2,1.5) [basept,label={[xshift=0pt, yshift = -15 pt] $p_{8}$}] {};
			\draw [line width = 0.8pt, ->] (p2) -- (p1);
			\draw [line width = 0.8pt, ->] (p4) -- (p3);
			\draw [line width = 0.8pt, ->] (p6) -- (p5);
			\draw [line width = 0.8pt, ->] (p8) -- (p7);
		\end{scope}
	
		\draw [->] (6.5,1.5)--(4.5,1.5) node[pos=0.5, below] {$\operatorname{Bl}_{p_{1}\cdots p_{8}}$};
	
		\begin{scope}[xshift = 8.5cm]
			\draw [blue, line width = 1pt] 	(-0.5,2.5) 	-- (3.6,2.5)	node [right] {$H_{g}-E_{5} - E_{7}$} node[pos=0, right] {};%{$g=\infty$};
			\draw [blue, line width = 1pt] 	(2.6,3) -- (2.6,-0.5)			node [below] {$H_{f}-E_{1}-E_{3}$} node[pos=0, above, xshift=7pt] {};%{$f=\infty$};

			\draw [blue,line width = 1pt] (2.2,0.8) -- (4.4,0.8) node [right] {$E_{1} - E_{2}$};
			\draw [red,line width = 1pt] (3.3,1.3) -- (3.7,0.3) node [right] {$E_{2}$};
			\draw [blue,line width = 1pt] (2.2,1.8) -- (4.4,1.8) node [right] {$E_{3} - E_{4}$};
			\draw [red,line width = 1pt] (3.3,2.3) -- (3.7,1.3) node [right] {$E_{4}$};
			\draw [blue,line width = 1pt] (1.6,0.7) -- (1.6,2.9) node [pos=0,below] {$\phantom{E}E_{5} - E_{6}$};
			\draw [red,line width = 1pt] (0.9,1.1) -- (1.9,1.5) node [pos=0,left] {$E_{6}$};
			\draw [blue,line width = 1pt] (0.2,0.7) -- (0.2,2.9) node [pos=0,below] {$E_{7} - E_{8}\phantom{E}$};
			\draw [red,line width = 1pt] (-0.5,1.3) -- (0.5,1.7) node [pos=0,left] {$E_{8}$};
		\end{scope}
	\end{tikzpicture}
	\caption{The standard $D_{5}^{(1)}$ Sakai surface.}
	\label{fig:surface-d5}
\end{figure}

We now describe the birational representation of the extended affine Weyl symmetry group $\widetilde{W}\left(A_{3}^{(1)}\right) = \operatorname{Aut}\left(A_{3}^{(1)}\right) \ltimes W\left(A_{3}^{(1)}\right)$. The affine Weyl group $W\left(A_{3}^{(1)}\right)$ is characterized by its generators $w_{j}$ and the relations encoded in the right part of Figure~\ref{fig:d-roots-d51-a31},
\begin{equation*}
	W\left(A_{3}^{(1)}\right) = W\left(\raisebox{-20pt}{\begin{tikzpicture}[
			elt/.style={circle,draw=black!100,thick, inner sep=0pt,minimum size=1.5mm}]
		\path 	(-0.5,0.5) 	node 	(a0) [elt, label={[xshift=-10pt, yshift = -10 pt] $\alpha_{0}$} ] {}
		        (-0.5,-0.5) node 	(a1) [elt, label={[xshift=-10pt, yshift = -10 pt] $\alpha_{1}$} ] {}
		        (0.5,-0.5) 	node 	(a2) [elt, label={[xshift=10pt, yshift = -10 pt] $\alpha_{2}$} ] {}
		        (0.5,0.5) 	node 	(a3) [elt, label={[xshift=10pt, yshift = -10 pt] $\alpha_{3}$} ] {};
		\draw [black,line width=1pt ] (a0) -- (a1) -- (a2)--(a3) -- (a0);
	\end{tikzpicture}} \right)
	=
	\left\langle w_{0},\dots, w_{3}\ \left|\
	\begin{alignedat}{2}
    w_{j}^{2} = e,\quad  w_{j}\circ w_{k} &= w_{k}\circ w_{j}& &\text{ when
   				\raisebox{-0.08in}{\begin{tikzpicture}[
   							elt/.style={circle,draw=black!100,thick, inner sep=0pt,minimum size=1.5mm}]
   						\path   ( 0,0) 	node  	(ai) [elt] {}
   						        ( 0.5,0) 	node  	(aj) [elt] {};
   						\draw [black] (ai)  (aj);
   							\node at ($(ai.south) + (0,-0.2)$) 	{$\alpha_{j}$};
   							\node at ($(aj.south) + (0,-0.2)$)  {$\alpha_{k}$};
   							\end{tikzpicture}}}\\
    w_{j}\circ w_{k}\circ w_{j} &= w_{k}\circ w_{j}\circ w_{k}& &\text{ when
   				\raisebox{-0.17in}{\begin{tikzpicture}[
   							elt/.style={circle,draw=black!100,thick, inner sep=0pt,minimum size=1.5mm}]
   						\path   ( 0,0) 	node  	(ai) [elt] {}
   						        ( 0.5,0) 	node  	(aj) [elt] {};
   						\draw [black] (ai) -- (aj);
   							\node at ($(ai.south) + (0,-0.2)$) 	{$\alpha_{j}$};
   							\node at ($(aj.south) + (0,-0.2)$)  {$\alpha_{k}$};
   							\end{tikzpicture}}}
	\end{alignedat}\right.\right\rangle,
\end{equation*}
and the group $\operatorname{Aut}\left(A_{3}^{(1)}\right)$ of Dynkin diagram automorphisms is isomorphic to the dihedral group $\mathbb{D}_4$\textemdash the symmetry group of the square\textemdash generated by three reflections $\sigma_{1}$, $\sigma_{2}$ and $\sigma_{3}$ which act on both the symmetry roots and surface roots,
\begin{equation}
		\sigma_{1} = (\alpha_{0}\alpha_{3})(\alpha_{1}\alpha_{2})=
		(\delta_{0}\delta_{5})(\delta_{1}\delta_{4})(\delta_{2}\delta_{3}),\quad
		\sigma_{2} = (\alpha_{0}\alpha_{2})=(\delta_{0}\delta_{1}), \quad
		\sigma_{3} = (\alpha_{1}\alpha_{3})=(\delta_{4}\delta_{5}),
	\end{equation}
here, we have employed the standard cycle notation for permutations. This group can be realized via reflections on the $\operatorname{Pic(\mathcal{X})}$, with $w_{j}$ expressed in terms of reflections through the symmetry roots $\alpha_{j}$,
\begin{equation}\label{eq:root-refl}
	w_{j}(\mathcal{C}) = w_{\alpha_{j}}(\mathcal{C}) = \mathcal{C} - 2
	\frac{\mathcal{C}\bullet \alpha_{j}}{\alpha_{j}\bullet \alpha_{j}}\alpha_{j}
	= \mathcal{C} + \left(\mathcal{C}\bullet \alpha_{j}\right) \alpha_{j},\qquad \mathcal{C}\in \operatorname{Pic(\mathcal{X})},
\end{equation}
and $\sigma_{j}$ can be realized as compositions of reflections in other roots in the $\operatorname{Pic(\mathcal{X})}$,
\begin{equation*}
		\sigma_{1} = w_{\mathcal{E}_{1}-\mathcal{E}_{7}}\circ w_{\mathcal{E}_{2}-\mathcal{E}_{8}}\circ w_{\mathcal{E}_{3}-\mathcal{E}_{5}}\circ w_{\mathcal{E}_{4}-\mathcal{E}_{6}}\circ w_{\mathcal{H}_{f}-\mathcal{H}_{g}},\quad
		\sigma_{2} = w_{\mathcal{E}_{1}-\mathcal{E}_{3}}\circ w_{\mathcal{E}_{2}-\mathcal{E}_{4}}, \quad
		\sigma_{3} = w_{\mathcal{E}_{5}-\mathcal{E}_{7}}\circ w_{\mathcal{E}_{6}-\mathcal{E}_{8}}.
	\end{equation*}
\begin{lemma}
The generators of the extended affine Weyl group $\widetilde{W}\left(A_{3}^{(1)}\right)$ transform an initial point configuration
\begin{equation*}
\left(\begin{matrix} a_{0} & a_{1} \\ a_{2} & a_{3} \end{matrix}\ ;\ t\ ;
		\begin{matrix} f \\ g \end{matrix}\right)
\end{equation*}
is given by the following birational maps:
\begin{alignat}{2}
		w_{0}&:
		\left(\begin{matrix} a_{0} & a_{1} \\ a_{2} & a_{3} \end{matrix}\ ;\ t\ ;
		\begin{matrix} f \\ g \end{matrix}\right)
		&&\mapsto
		\left(\begin{matrix} -a_{0} & a_{0} + a_{1} \\ a_{2} & a_{0} + a_{3} \end{matrix}\ ;\ t\ ;
		\begin{matrix} \displaystyle f + \frac{a_{0}}{g + t} \\ g \end{matrix}\right), \\
		w_{1}&: \left(\begin{matrix} a_{0} & a_{1} \\ a_{2} & a_{3} \end{matrix}\ ;\ t\ ;
		\begin{matrix} f \\ g \end{matrix}\right)
		&&\mapsto
		\left(\begin{matrix} a_{0} + a_{1} & -a_{1} \\ a_{1} + a_{2} & a_{3} \end{matrix}\ ;\ t\ ;
		\begin{matrix}  f \\ \displaystyle g - \frac{a_{1}}{f} \end{matrix}\right), \\
		w_{2}&:
		\left(\begin{matrix} a_{0} & a_{1} \\ a_{2} & a_{3} \end{matrix}\ ;\ t\ ;
		\begin{matrix} f \\ g \end{matrix}\right)
		&&\mapsto
		\left(\begin{matrix} a_{0} & a_{1} + a_{2} \\ -a_{2} & a_{2} + a_{3} \end{matrix}\ ;\ t\ ;
		\begin{matrix} \displaystyle f + \frac{a_{2}}{g}\\ g \end{matrix}\right), \\
		w_{3}&:
		\left(\begin{matrix} a_{0} & a_{1} \\ a_{2} & a_{3} \end{matrix}\ ;\ t\ ;
		\begin{matrix} f \\ g \end{matrix}\right)
		&&\mapsto
		\left(\begin{matrix} a_{0}+a_{3} & a_{1} \\ a_{2}+a_{3} & -a_{3} \end{matrix}\ ;\ t\ ;
		\begin{matrix} f \\ \displaystyle  g - \frac{a_{3}}{f-1} \end{matrix}\right), \\
        \sigma_{1}&:
		\left(\begin{matrix} a_{0} & a_{1} \\ a_{2} & a_{3} \end{matrix}\ ;\ t\ ;
		\begin{matrix} f \\ g \end{matrix}\right)
		&&\mapsto
		\left(\begin{matrix} a_{3} & a_{2} \\ a_{1} & a_{0}  \end{matrix}\ ;\ -t\ ;
		\begin{matrix} \displaystyle -\frac{g}{ t} \\ f t \end{matrix}\right), \\
		\sigma_{2}&:
		\left(\begin{matrix} a_{0} & a_{1} \\ a_{2} & a_{3} \end{matrix}\ ;\ t\ ;
		\begin{matrix} f \\ g \end{matrix}\right)
		&&\mapsto
		\left(\begin{matrix} a_{2} & a_{1} \\ a_{0} & a_{3}  \end{matrix}\ ;\ -t\ ;
		\begin{matrix} f \\ g + t \end{matrix}\right), \\
		\sigma_{3}&:
		\left(\begin{matrix} a_{0} & a_{1} \\ a_{2} & a_{3} \end{matrix}\ ;\ t\ ;
		\begin{matrix} f \\ g \end{matrix}\right)
		&&\mapsto
		\left(\begin{matrix} a_{0} & a_{3} \\ a_{2} & a_{1}  \end{matrix}\ ;\ -t\ ;
		\begin{matrix} 1-f \\ -g  \end{matrix}\right).
	\end{alignat}
\end{lemma}
\begin{proof}
The proof proceeds along the same lines as in \cite{DzhFilSto:2020:RCDOPHWDPE} and \cite{DzhTak:2018:OSAOSGTODPE}, so we omit it here.
\end{proof}
The standard discrete Painlev\'e equation on the $D_{5}^{(1)}$ surface is given in section 8.1.17 of \cite{KajNouYam:2017:GAOPE} as
\begin{equation}\label{eq:dPD5-KNY}
	\overline{f} + f = 1 - \frac{a_{2}}{g} - \frac{a_{0}}{g+t},\qquad
	g + \underline{g} = -t + \frac{a_{1}}{f} + \frac{a_{3}}{f-1},
\end{equation}
with the root variable evolution and normalization given by
\begin{equation}\label{eq:dPD5-rv-evol}
	\overline{a}_{0} = a_{0} + 1, \quad \overline{a}_{1} = a_{1}-1, \quad \overline{a}_{2} = a_{2} + 1,\quad \overline{a}_{3} = a_{3} - 1,\qquad
	 a_{0} + a_{1} + a_{2}  + a_{3} = 1.
\end{equation}
According to the evolution of the root variables \eqref{eq:dPD5-rv-evol} we can see the corresponding translation on the root lattice is
\begin{equation}\label{eq:dPD5-transl-KNY}
	\phi_{*}: \upalpha =  \langle \alpha_{0}, \alpha_{1}, \alpha_{2}, \alpha_{3} \rangle
	\mapsto \phi_{*}(\upalpha) = \upalpha + \langle -1,1,-1,1 \rangle \delta.
\end{equation}
Employing established methodologies (see detail in \cite{DzhTak:2018:OSAOSGTODPE}), we derive the subsequent decomposition of $\phi$ using the generators of $\widetilde{W}\left(A_{3}^{(1)}\right)$,
\begin{equation}\label{eq:dPD5-decomp-KNY}
	\phi = \sigma_{3}\sigma_{2} w_{3} w_{1} w_{2} w_{0}.
\end{equation}

Through a systematic application of the reduction approach detailed in \cite{DzhFilSto:2020:RCDOPHWDPE}, we show that the recurrence relation given by \eqref{eq:xyn-evol} is a discrete Painlev\'e equation that is equivalent to the standard d-$\dPain{A_{3}^{(1)}/D_{5}^{(1)}}$ equation \eqref{eq:dPD5-KNY}. The main result of our work is stated below.
\begin{theorem}\label{thm:coordinate-change}
The recurrence relation \eqref{eq:xyn-evol} is equivalent to the standard discrete Painlev\'{e} equation~\eqref{eq:dPD5-KNY} presented in \cite{KajNouYam:2017:GAOPE}. This equivalence is established through the following variable transformation:
\begin{equation}\label{eq:xy2qp}
x(f,g) =-\frac{f(g+t)+n}{t(fg+n)},\quad
y(f,g) = \frac{(fg+n)(g+t)}{t},\quad
s(t)=-t.
\end{equation}
The inverse transformation of variables is provided by
\begin{equation}\label{eq:qp2xy}
f(x,y) = \frac{(1-sx)(n-y+sxy)}{s^2x},\quad
g(x,y) = \frac{s(y-n)}{(1-sx)y-n},\quad
t(s)=-s.
\end{equation}
\end{theorem}

%section Preliminaries_and_the_main_result (end)

\section{Proof of the main result} % (fold)
\label{Proof_of_the_main_result}
To proof the Theorem~\ref{thm:coordinate-change}, the first step is understand the singularity structure of the mapping \eqref{eq:xyn-evol}. At this process closely follows the detailed procedure outlined in \cite{DzhFilSto:2020:RCDOPHWDPE}, we omit most of the computations and just present the results.

Using the standard notation $x:=x_{n}$, $\overline{x}: = x_{n+1}$, $\underline{x}: =x_{n-1}$ and similarly for $y$. The recurrence \eqref{eq:xyn-evol} naturally defines two mappings, the forward mapping
\begin{equation}
\varphi_1^{(n)}:(x,y)\mapsto(\overline{x},y)=(\frac{y^2-(2n+\beta)y+n(n+\beta)}{s^2x(y^2+\alpha y)},y),
\end{equation}
and the backward mapping
\begin{equation}
\varphi_2^{(n)}:(x,y)\mapsto(x,\underline{y})=(x,-y - \frac{\alpha s^2x^2+s(2n-1-\alpha+\beta+s)x-2n-\beta+1}{(1-sx)^2}).
\end{equation}
Then we get the full forward and backward mappings, i.e. $\varphi^{(n)} = \left(\varphi_{2}^{(n+1)}\right)^{-1} \circ \varphi_{1}^{(n)}: (x,y)\mapsto (\overline{x},\overline{y})$ and $\varphi^{(n)} = \left(\varphi_{1}^{(n-1)}\right)^{-1} \circ \varphi_{2}^{(n)}: (x,y)\mapsto (\underline{x},\underline{y})$. The explicit forms of these mappings are complex, so we omit them.

Extending these mappings from $\mathbb{C}\times\mathbb{C}$ to $\mathbb{P}^1\times\mathbb{P}^1$, we obtain the base points of the mappings,
\begin{alignat}{2}\label{eq-base-pt-te-jacobi}
	&q_{1}(x = 0,y = n),&\quad &q_{2}(x = 0,y = n+\beta),\notag\\
    &q_{3}\left(X = \frac{1}{x} = 0,y = 0\right),&\quad &q_{4}\left(X = \frac{1}{x} = 0,y = -\alpha\right),\notag\\
	&q_{5}\left(X = \frac{1}{x}=s,Y = \frac{1}{y}= 0\right)&\leftarrow
	&q_{6}\left(u_{5} = X-s = 0,v_{5}= \frac{Y}{X-s} = 0\right)\\
	&&\leftarrow &q_{7}\left(U_{6} = \frac{u_5}{v_5} = -s^3,V_{6} = v_5 = 0\right)\notag\\
	&&\leftarrow &q_{8}\left(U_{7}=\frac{U_6+s^3}{V_6} = s^4(1 - 2n + s - \alpha-\beta),
 V_{7} = V_{6} = 0\right).\notag
\end{alignat}
These points are shown on the left side of Figure~\ref{fig:surface-te-jacobi}. By employing the blowup procedure (for details, refer to \cite{Sha:2013:BAG1}, \cite{DzhFilSto:2020:RCDOPHWDPE}), we obtain the family of Sakai surfaces $\mathcal{X}=\mathcal{X}_{\alpha,\beta,s,n}$, parameterized by $\alpha$, $\beta$, $s$ and $n$\textemdash also known as the configuration space of the mappings\textemdash as depicted in Figure~\ref{fig:surface-te-jacobi} (right). It is worthy noting that the configuration (Figure~\ref{fig:surface-te-jacobi}) of the blowup points lies on a bi-quadratic curve given by the equation $x=0$ in the affine $(x,y)$-chart. Moreover, it is straightforward to see that these points lie on the polar divisor of a symplectic form, which in the affine $(x,y)$-chart is given by $\omega = k \frac{\mathrm{d}x\wedge \mathrm{d}y}{x}$.
\begin{figure}[ht]
	\begin{tikzpicture}[>=stealth,basept/.style={circle, draw=red!100, fill=red!100, thick, inner sep=0pt,minimum size=1.2mm},scale=0.9]
		\begin{scope}[xshift = 0cm]
			\draw [black, line width = 1pt]  	(4,0) 	-- (-0.5,0) 	node [left] {$H_{y}$} node[pos=0, right] {$y=0$};
			\draw [black, line width = 1pt] 	(4,2.5) -- (-0.5,2.5)	node [left] {$H_{y}$} node[pos=0, right] {$y=\infty$};
			\draw [black, line width = 1pt] 	(0,3) -- (0,-0.5)		node [below] {$H_{x}$} node[pos=0, above, xshift=-7pt] {$x=0$};
			\draw [black, line width = 1pt] 	(3,3) -- (3,-0.5)		node [below] {$H_{x}$} node[pos=0, above, xshift=7pt] {$x=\infty$};
			
			\node (q1) at (0,1.1) 	[basept,label={[xshift=-10pt, yshift = -10 pt] $q_{1}$}] {};
			\node (q2) at (0,1.7) 	[basept,label={[xshift=-10pt, yshift = -10 pt] $q_{2}$}] {};
			\node (q3) at (3,0) [basept,label={[xshift=10pt, yshift =-5 pt] $q_{3}$}] {};
			\node (q4) at (3,0.9) [basept,label={[xshift=10pt, yshift =-5 pt] $q_{4}$}] {};
			\node (q5) at (1,2.5) 	[basept,label={[above ] $q_{5}$}] {};
			\node (q6) at (1.45,3.0) 	[basept,label={[above ] $q_{6}$}] {};
			\node (q7) at (2,3.0)		[basept,label={[above ] $q_{7}$}] {};
			\node (q8) at (2.55,3.0) [basept,label={[above ] $q_{8}$}] {};
			\draw [line width = 0.8pt, ->] (q8) edge (q7) (q7) edge (q6) (q6) edge (q5) ;
		\end{scope}

		\draw [->] (7,1.5)--(5,1.5) node[pos=0.5, below] {$\operatorname{Bl}_{q_{1}\cdots q_{8}}$};

		\begin{scope}[xshift = 9cm]
			\draw [blue, line width = 1pt] 	(-0.5,2.5) -- (4,2.5)	node [right] {$H_{y} - F_{5}-F_{6}$};
			\draw [blue, line width = 1pt] 	(0,3) -- (0,-0.5)		node [below] {$H_{x}-F_{1}-F_{2}$};
            \draw [red, line width = 1pt] 	(2.8,0) -- (-0.5,0)	node [left] {$H_{y} - F_{3}$};
			\draw [blue, line width = 1pt] 	(3.5,3) -- (3.5,0.3)  node [right] {$H_{x} - F_{3}-F_{4}$};

			\draw [red, line width = 1 pt] (-0.4,0.9) -- (0.4,1.3) node [pos = 0, left] {$F_{1}$};
			\draw [red, line width = 1 pt] (-0.4,1.5) -- (0.4,1.9) node [pos = 0, left] {$F_{2}$};
            \draw [red, line width = 1 pt] (3.9,0.8) -- (1.9,-0.2) node [below] {$F_{3}$};
            \draw [red, line width = 1 pt] (3.1,0.7) -- (3.9,1.1) node [right] {$F_{4}$};
			\draw [blue, line width = 1 pt] (1,2.9) -- (1,0.5)  node [pos = 0, above right] {$F_{6}-F_{7}$};
            \draw [blue, line width = 1 pt] (0.7,1.1) -- (1.6,1.1)  node [right] {$F_{5}-F_{6}$};
			\draw [blue, line width = 1 pt] (0.7,1.9) -- (2.8,1.9)  node [above] {$F_{7}-F_{8}$};
			\draw [red, line width = 1 pt] (1.8,2.3) -- (1.8,1.5)  node [right] {$F_{8}$};
		\end{scope}
	\end{tikzpicture}
	\caption{The Sakai surface for the time-evolved Jacobi weight recurrence.}
	\label{fig:surface-te-jacobi}
\end{figure}
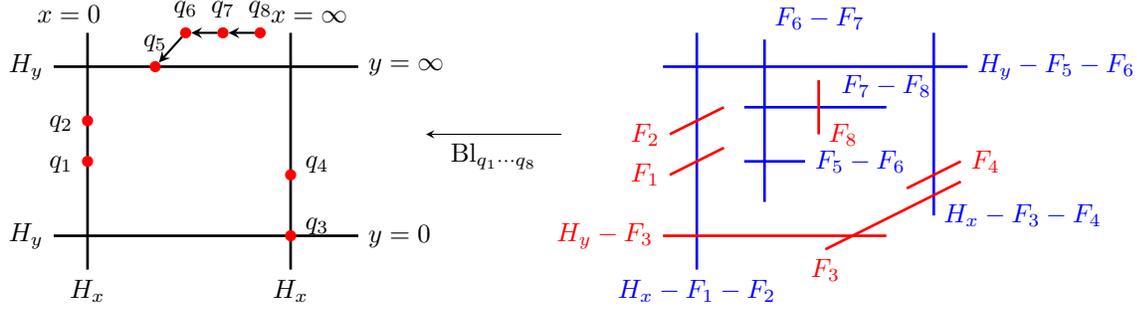

A key object associated with this family is the Picard lattice $\operatorname{Pic}(\mathcal{X}) = \operatorname{Span}_{\mathbb{Z}}\{\mathcal{H}_{x},\mathcal{H}_{y},\mathcal{F}_{1},\ldots, \mathcal{F}_{8}\}$ generated by the classes of coordinate lines $\mathcal{H}_{x,y}=[H_{x,y=c}]$, and classes of the  exceptional divisors $\mathcal{F}_i=[F_i]$. This lattice is equipped with the intersection product defined  on the generators by
\begin{equation}\label{eq:int-form}
\mathcal{H}_{x}\bullet \mathcal{H}_{x} = \mathcal{H}_{y}\bullet \mathcal{H}_{y} = \mathcal{H}_{x}\bullet \mathcal{F}_{i} =
\mathcal{H}_{y}\bullet \mathcal{F}_{j} = 0,\qquad \mathcal{H}_{x}\bullet \mathcal{H}_{y} = 1,\qquad  \mathcal{F}_{i}\bullet \mathcal{F}_{j} = - \delta_{ij}.	
\end{equation}
Using this inner product, we are able to assign to each curve on $\mathcal{X}$ its self-intersection index. Looking at the configuration of the $-2$-curves (depicted as blue lines in Figure~\ref{fig:surface-te-jacobi}), which form irreducible components of the anti-canonical divisor
\begin{align*}
-\mathcal{K}_{\mathcal{X}}
&=2\mathcal{H}_{x}+2\mathcal{H}_{y}-\mathcal{F}_{1}-\cdots-\mathcal{F}_{8}\\
&= [H_{x}-F_{1}-F_{2}]+[H_{x}-F_{3}-F_{4}]+2[H_{y}-F_{5}-F_{6}]+2[F_{6}-F_{7}]+[F_{5}-F_{6}]+[F_{7}-F_{8}]\\
&=
\delta_{0}+\delta_{1}+2\delta_{2}+2\delta_{3}+\delta_{4}+\delta_{5},
\end{align*}
then we can readily see that the surface type of our recurrence is $D_{5}^{(1)}$, with the surface root basis shown on Figure~\ref{fig:d-roots-lw}.
\begin{figure}[H]
\begin{equation}\label{eq:d-roots-lw}
	\raisebox{-32.1pt}{\begin{tikzpicture}[
			elt/.style={circle,draw=black!100,thick, inner sep=0pt,minimum size=2mm},scale=0.75]
		\path 	(-1,1) 	node 	(d0) [elt, label={[xshift=-10pt, yshift = -10 pt] $\delta_{0}$} ] {}
		        (-1,-1) node 	(d1) [elt, label={[xshift=-10pt, yshift = -10 pt] $\delta_{1}$} ] {}
		        ( 0,0) 	node  	(d2) [elt, label={[xshift=-10pt, yshift = -10 pt] $\delta_{2}$} ] {}
		        ( 1,0) 	node  	(d3) [elt, label={[xshift=10pt, yshift = -10 pt] $\delta_{3}$} ] {}
		        ( 2,1) 	node  	(d4) [elt, label={[xshift=10pt, yshift = -10 pt] $\delta_{4}$} ] {}
		        ( 2,-1) node 	(d5) [elt, label={[xshift=10pt, yshift = -10 pt] $\delta_{5}$} ] {};
		\draw [black,line width=1pt ] (d0) -- (d2) -- (d1)  (d2) -- (d3) (d4) -- (d3) -- (d5);
	\end{tikzpicture}} \qquad
			\begin{alignedat}{2}
			\delta_{0} &= \mathcal{H}_{x} - \mathcal{F}_{1} - \mathcal{F}_{2}, &\qquad  \delta_{3} &= \mathcal{F}_{6} - \mathcal{F}_{7},\\
			\delta_{1} &= \mathcal{H}_{x} - \mathcal{F}_{3} - \mathcal{F}_{4}, &\qquad  \delta_{4} &= \mathcal{F}_{5} - \mathcal{F}_{6},\\
			\delta_{2} &= \mathcal{H}_{y} - \mathcal{F}_{5} - \mathcal{F}_{6}, &\qquad  \delta_{5} &= \mathcal{F}_{7} - \mathcal{F}_{8}.
			\end{alignedat}
\end{equation}
	\caption{The surface root basis for the time-evolved Jacobi weight recurrence.}
	\label{fig:d-roots-lw}
\end{figure}

We can calculate the evolution of the surface roots \eqref{eq:d-roots-lw} under the full forward mapping $\varphi_{*}: \operatorname{Pic}(\mathcal{X})\to \operatorname{Pic}(\overline{\mathcal{X}})$, with $\overline{\mathcal{X}}=\mathcal{X}_{\alpha,\beta,s,n+1}$, as summarized in the following lemma.
\begin{lemma}\label{lem:dyn}
	The action of the mapping $\varphi_{*}: \operatorname{Pic}(\mathcal{X})\to \operatorname{Pic}(\overline{\mathcal{X}})$
is given by
	\begin{align*}
\begin{aligned}
&\mathcal{H}_x\mapsto5\overline{\mathcal{H}}_x+2\overline{\mathcal{H}}_y-\overline{\mathcal{F}}_{1234}-2\overline{\mathcal{F}}_{5678}, \\
&\mathcal{F}_1\mapsto2\overline{\mathcal{H}}_x+\overline{\mathcal{H}}_y-\overline{\mathcal{F}}_{25678},
\\
&\mathcal{F}_2\mapsto2\overline{\mathcal{H}}_x+\overline{\mathcal{H}}_y-\overline{\mathcal{F}}_{15678},
\\
&\mathcal{F}_3\mapsto2\overline{\mathcal{H}}_x+\overline{\mathcal{H}}_y-\overline{\mathcal{F}}_{45678},
\\
&\mathcal{F}_4\mapsto2\overline{\mathcal{H}}_x+\overline{\mathcal{H}}_y-\overline{\mathcal{F}}_{35678},
\end{aligned}
&&
\begin{aligned}
&\mathcal{H}_y\mapsto2\overline{\mathcal{H}}_x+\overline{\mathcal{H}}_y-\overline{\mathcal{F}}_{5678}, \\
&\mathcal{F}_5\mapsto\overline{\mathcal{H}}_x-\overline{\mathcal{F}}_{8},
\\
&\mathcal{F}_6\mapsto\overline{\mathcal{H}}_x-\overline{\mathcal{F}}_{7},
\\
&\mathcal{F}_7\mapsto\overline{\mathcal{H}}_x-\overline{\mathcal{F}}_{6},
\\
&\mathcal{F}_8\mapsto\overline{\mathcal{H}}_x-\overline{\mathcal{F}}_{5}.
\end{aligned}
\end{align*}
where we use the notation $\mathcal{F}_{i\cdots j}=\mathcal{F}_{i}+ \cdots + \mathcal{F}_{j}$.
\end{lemma}

The next step is to match \eqref{eq:xyn-evol} and \eqref{eq:dPD5-KNY} at the geometry and dynamic levels. First, we need to match the geometries by finding a basis transformation of the Picard lattice that identifies the surface roots in Figure~\ref{fig:d-roots-lw} and \ref{fig:d-roots-d51-a31} (left). We immediately find that such a transformation is given by
\begin{equation}\label{eq:pre-basis-transformation}
\begin{alignedat}{5}	
\mathcal{H}_f&=2\mathcal{H}_x+\mathcal{H}_y-\mathcal{F}_{1356},&\quad
\mathcal{E}_1&=\mathcal{H}_x-\mathcal{F}_1,	&\quad
\mathcal{E}_2&=\mathcal{F}_2,&\quad
\mathcal{E}_3&=\mathcal{H}_x-\mathcal{F}_3,	&\quad
\mathcal{E}_4&=\mathcal{F}_4,\\
\mathcal{H}_g&=\mathcal{H}_x,&\quad
\mathcal{E}_5&=\mathcal{H}_x-\mathcal{F}_6,&\quad
\mathcal{E}_6&=\mathcal{H}_x-\mathcal{F}_5,&\quad
\mathcal{E}_7&=\mathcal{F}_7,&\quad
\mathcal{E}_8&=\mathcal{F}_8.
\end{alignedat}
\end{equation}
Using this identification, along with the standard symmetry root basis shown in Figure~\ref{fig:d-roots-d51-a31} (right), we get the preliminary symmetry root basis for the time-evolved Jacobi recurrence, which is given by
\begin{equation}	
\alpha_{0}=\mathcal{F}_{1}-\mathcal{F}_{2},\quad
\alpha_{1}=\mathcal{H}_{y}-\mathcal{F}_{13},\quad
\alpha_{2}=\mathcal{F}_{3}-\mathcal{F}_{4},\quad
\alpha_{3}=2\mathcal{H}_{x}+\mathcal{H}_{y}-\mathcal{F}_{135678},
\end{equation}
and applying the evolution in Lemma~\ref{lem:dyn}, we see that the symmetry roots evolve as
\begin{equation}
\varphi_{*}: \upalpha =  \langle \alpha_{0}, \alpha_{1}, \alpha_{2}, \alpha_{3}  \rangle
	\mapsto \varphi_{*}(\upalpha) = \upalpha + \langle 0,-1,0,1 \rangle \delta,
\end{equation}
decompose it in terms of the generators of the extended affine Weyl symmetry group, we get
\begin{equation}\label{eq:generators}
\varphi = \sigma_{3}\sigma_{2}w_{1}w_{2}w_{0}w_{1},
\end{equation}
compare \eqref{eq:generators} and \eqref{eq:dPD5-decomp-KNY}, we immediately see that $\varphi=w_{1}\circ\phi\circ w_{1}^{-1}$ (note that $w_{1} \sigma_{3}\sigma_{2}  = \sigma_{3}\sigma_{2} w_{3}$ and that $w_{1}$ is an involution, $w_{1}^{-1} = w_{1}$). Thus, after adjusting our basis transformation \eqref{eq:pre-basis-transformation} in $\operatorname{Pic}(\mathcal{X})$ by acting on it with $w_{1}$, our dynamic becomes equivalent to the standard
equation \eqref{eq:dPD5-KNY}.
\begin{lemma}\label{lem:change-basis-fin}
The basis transformation of the Picard lattice identifying both the geometry and the
standard dynamics between the time-evolved Jacobi recurrence and the standard d-$\dPain{A_{3}^{(1)}/D_{5}^{(1)}}$ equation is given by
	\begin{align*}
\begin{aligned}
&\mathcal{H}_x=\mathcal{H}_f+\mathcal{H}_g-\mathcal{E}_5-\mathcal{E}_6,
\\
&\mathcal{H}_y=\mathcal{H}_f+2\mathcal{H}_g-\mathcal{E}_1-\mathcal{E}_3-\mathcal{E}_5-\mathcal{E}_6,
\\
&\mathcal{F}_1=\mathcal{H}_f+\mathcal{H}_g-\mathcal{E}_1-\mathcal{E}_5-\mathcal{E}_6,
\\
&\mathcal{F}_2=\mathcal{E}_2,
\\
&\mathcal{F}_3=\mathcal{H}_f+\mathcal{H}_g-\mathcal{E}_3-\mathcal{E}_5-\mathcal{E}_6,
\\
&\mathcal{F}_4=\mathcal{E}_4,
\\
&\mathcal{F}_5=\mathcal{H}_g-\mathcal{E}_6,
\\
&\mathcal{F}_6=\mathcal{H}_g-\mathcal{E}_5,
\\
&\mathcal{F}_7=\mathcal{E}_7,
\\
&\mathcal{F}_8=\mathcal{E}_8,
\end{aligned}
&&
\begin{aligned}
&\mathcal{H}_f=2\mathcal{H}_x+\mathcal{H}_y-\mathcal{F}_1-\mathcal{F}_3-\mathcal{F}_5-\mathcal{F}_6,
\\
&\mathcal{H}_g=\mathcal{H}_x+\mathcal{H}_y-\mathcal{F}_1-\mathcal{F}_3,
\\
&\mathcal{E}_1=\mathcal{H}_x-\mathcal{F}_1,
\\
&\mathcal{E}_2=\mathcal{F}_2,
\\
&\mathcal{E}_3=\mathcal{H}_x-\mathcal{F}_3,
\\
&\mathcal{E}_4=\mathcal{F}_4,
\\
&\mathcal{E}_5=\mathcal{H}_x+\mathcal{H}_y-\mathcal{F}_1-\mathcal{F}_3-\mathcal{F}_6,
\\
&\mathcal{E}_6=\mathcal{H}_x+\mathcal{H}_y-\mathcal{F}_1-\mathcal{F}_3-\mathcal{F}_5,
\\
&\mathcal{E}_7=\mathcal{F}_7,
\\
&\mathcal{E}_8=\mathcal{F}_8.
\end{aligned}
\end{align*}
The resulting identification of the symmetry root bases is given by
\begin{equation}\label{eq:a-roots-lw-fin}
			\alpha_{0} = \mathcal{H}_{y} - \mathcal{F}_{23}, \quad
            \alpha_{1} = \mathcal{F}_{13} - \mathcal{H}_{y}, \quad
			\alpha_{2} = \mathcal{H}_{y} -  \mathcal{F}_{14},\quad
			\alpha_{3} = 2\mathcal{H}_{x} + \mathcal{H}_{y} -
			\mathcal{F}_{135678}.
\end{equation}
\end{lemma}

The last step is to find the explicit coordinate transformation. To do this, we need to calculate the root variables for the time-evolved Jacobi case, which is done in the following lemma.
\begin{lemma}
	\qquad
	
	\begin{enumerate}[(i)]
		\item The residues of the symplectic form $\omega = k \frac{\mathrm{d}x\wedge \mathrm{d}y}{x}$
		along the irreducible components of the polar divisor are given by
		\begin{alignat*}{3}
			\operatorname{res}_{d_{0}} \omega &=  k\mathrm{d}y, &\qquad
			\operatorname{res}_{d_{2}} \omega &=  0,\quad &\qquad
			\operatorname{res}_{d_{4}} \omega &=  k\frac{\mathrm{d}v_5}{sv_{5}^2},\\
			\operatorname{res}_{d_{1}} \omega &=  -k\mathrm{d}y, &\qquad
            \operatorname{res}_{d_{3}} \omega &=  k\frac{\mathrm{d}U_6}{s^2}, &\qquad
			\operatorname{res}_{d_{5}} \omega &=  k\frac{\mathrm{d}U_7}{s^4}.
		\end{alignat*}
		
		\item For the standard root variable normalization $\mathcal{X}(\delta)=a_{0}+a_{1}+a_{2}+a_{3}=1$ we need to take $k=-1$ and root variables $a_{i}$ are then given by

		\begin{equation}
a_{0} = n+\beta, \quad a_{1} = -n,\quad
a_{2} = n + \alpha, \quad a_{3}=1-n-\alpha-\beta.
\end{equation}
		\end{enumerate}
\end{lemma}

We are now finally able to derive the explicit coordinate transformation \eqref{eq:xy2qp} and \eqref{eq:qp2xy} in Theorem \ref{thm:coordinate-change}. This computation follows standard procedures, for detailed examples,refer to \cite{DzhFilSto:2020:RCDOPHWDPE} and \cite{DzhTak:2018:OSAOSGTODPE}. Here, we merely provide an outline.

From the basis transformation for the coordinate classes on the Picard lattice given by Lemma~\ref{lem:change-basis-fin},
\begin{equation*}
\mathcal{H}_x=\mathcal{H}_f+\mathcal{H}_g-\mathcal{E}_5-\mathcal{E}_6, \quad
\mathcal{H}_y=\mathcal{H}_f+2\mathcal{H}_g-\mathcal{E}_1-\mathcal{E}_3-\mathcal{E}_5-\mathcal{E}_6,
\end{equation*}
we see that $x$ and $y$ represent projective coordinates on pencils of $(1,1)$-curves in the $(f,g)$-coordinates. Specifically, $x$ corresponds to the pencil passing through $p_{5}$ and $p_{6}$, while $y$ corresponds to the pencil passing through $p_{1}$, $p_{3}$, $p_{5}$ and $p_{6}$. Consequently, we take coordinate transformation as
\begin{equation*}
x(f,g)=\frac{Af+B(fg-a_1)}{Cf+D(fg-a_1)}, \quad
y(f,g)=\frac{K+Lg(f(g+t)-a_1)}{M+Ng(f(g+t)-a_1)},
\end{equation*}
where the coefficients $A,\ldots,N$ are still to be determined. For example, the correspondence $H_{f} - E_{1} - E_{3} = H_{y}-F_{5} - F_{6}$ means that
\begin{equation*}
Y(F=0,g)=\frac{M\cdot0+Ng(g+t)-N\cdot0}{K\cdot0+Lg(g+t)-L\cdot0}=\frac{N}{L},\quad \text{and so}\quad N=0,
\end{equation*}
then we can take $M=1$ to get
\begin{equation*}
y(f,g)=K+Lg(f(g+t)-a_1).
\end{equation*}
Proceeding in the same way, we can proof the \eqref{eq:xy2qp} in Theorem~\ref{thm:coordinate-change}. The
inverse change of variables \eqref{eq:qp2xy} can be either obtained directly, or computed in the similar
way.
\section*{Acknowledgements} % (fold)
\label{sec:acknowledgements}
M. Zhu acknowledges the support of the National Natural Science Foundation of China under Grant No. 12201333, the Natural Science Foundation of Shandong Province (Grant No. ZR2021QA034), and the Breeding Plan of Shandong Provincial Qingchuang Research Team (Grant No. 2023KJ135). X. Zhang is grateful for the support provided by the Natural Science Foundation of Shandong Province (Grant No. ZR2022MA057).

% \small
% \bibliographystyle{amsxport}
\bibliographystyle{amsalpha}

\providecommand{\bysame}{\leavevmode\hbox to3em{\hrulefill}\thinspace}
\providecommand{\MR}{\relax\ifhmode\unskip\space\fi MR }
\newcommand{\etalchar}[1]{$^{#1}$}
\providecommand{\href}[2]{#2}

\end{document}